\documentclass[11pt]{amsart}
\usepackage{amsmath, amsthm, amscd, amsfonts, amssymb, graphicx, color,float,pgf,tikz}
\usepackage{amssymb,fontenc}
\usepackage{latexsym,wasysym,mathrsfs}
\usepackage{hyperref}
\usetikzlibrary{arrows}
\usepackage[square,numbers,sort&compress]{natbib}
\makeatletter \oddsidemargin.9375in \evensidemargin \oddsidemargin
\newtheorem{theorem}{Theorem}[section]
\newtheorem{lemma}[theorem]{Lemma}
\newtheorem{proposition}[theorem]{Proposition}
\newtheorem{corollary}[theorem]{Corollary}

\theoremstyle{definition}

\newtheorem{example}[theorem]{Example}

\newtheorem{remark}[theorem]{Remark}
\numberwithin{equation}{section}

\newcommand{\be}{\begin{equation}}
\newcommand{\ee}{\end{equation}}

\numberwithin{equation}{section}

\usepackage{etoolbox}
\makeatletter
\patchcmd{\@settitle}{\uppercasenonmath\@title}{}{}{}
\patchcmd{\@setauthors}{\MakeUppercase}{}{}{}
\makeatother

\allowdisplaybreaks
\begin{document}


\title[Cohomological properties  and Arens regularity of Banach algebras]{Cohomological properties  and Arens regularity of Banach algebras}

\author[Hossein Eghbali Sarai,  Kazem Haghnejad Azar and Ali Jabbari]{Hossein Eghbali Sarai$^1$,  Kazem Haghnejad Azar$^2$$^{*}$ and Ali Jabbari$^3$ }

\address{ $^1$ Department of Mathematics, University of Mohaghegh Ardabili, Ardabil, Iran}
\email{h.eghbali.sarai@gmail.com}

\address{ $^{2}$Department of Mathematics, University of Mohaghegh Ardabili, Ardabil, Iran .}
\email{haghnejad@uma.ac.ir}

\address{ $^{3}$Young Researchers and Elite Club, Islamic Azad University, Ardabil Branch, Ardabil, Iran.}
\email{jabbari\underline{ }al@yahoo.com}


\keywords{Arens regularity, topological centers, cohomological group, weakly amenable,   Connes-amenability.\\
\indent Received: dd mmmm yyyy,    Accepted: dd mmmm yyyy.
\\
\indent $^{*}$ Corresponding author}
\maketitle
\hrule width \hsize \kern 1mm


\begin{abstract}
 In this paper, we study some cohomlogical  properties of Banach algebras. For a Banach algebra $A$ and a Banach $A$-bimodule $B$, we investigate the vanishing of the first Hochschild cohomology groups $H^1(A^n,B^m)$ and $H_{w^*}^1(A^n,B^m)$, where $0\leq m,n\leq 3$. For amenable Banach algebra $A$, we show that there are Banach $A$-bimodules $C$, $D$ and  elements $\mathfrak{a}, \mathfrak{b}\in A^{**}$ such that $$Z^1(A,C^*)=\{R_{D^{\prime\prime}(\mathfrak{a})}:~D\in Z^1(A,C^*)\}=\{L_{D^{\prime\prime}(\mathfrak{b})}:~D\in Z^1(A,D^*)\}.$$
where, for every $b\in B$, $L_{b}(a)=ba$ and $R_{b}(a)=a b,$ for every $a\in A$. Moreover, under a condition, we show that if the second transpose of a continuous derivation from the Banach algebra $A$ into $A^*$ i.e., a  continuous linear map from $A^{**}$ into $A^{***}$, is a derivation, then $A$ is Arens regular. Finally, we show that if $A$ is a dual left strongly irregular Banach algebra such that its second dual is amenable, then $A$ is reflexive.
\end{abstract}
\maketitle
\vspace{0.1in}
\hrule width \hsize \kern 1mm

\section{Introduction}
A derivation from a Banach algebra  $A$ into a Banach $A-$bimodule $B$ is a bounded linear mapping $D:A\longrightarrow B$ such that $$D(ab)=aD(b)+D(a)b\qquad\text{for all}\qquad a, b\in A.$$

The space of continuous derivations from $A$ into $B$ is denoted by $Z^1(A,B)$. The easiest example of derivations is the inner derivations, which are given for each $b\in B$ by
$$\delta_b(a)=ab-ba\qquad\text{for all}\qquad a\in A.$$

The space of inner derivations from $A$ into $B$ is denoted by $B^1(A,B)$.
The Banach algebra $A$ is said to be  amenable, when for every Banach $A$-bimodule $B$, the inner derivations are only derivations existing from $A$ into $B^*$, in the other word, $H^1(A,B^*)=Z^1(A,B^*)/ B^1(A,B^*)=\{0\}$ and $A$ is said to be weakly amenable if $H^1(A,A^*)=\{0\}$.

The concept of amenability for a Banach algebra $A$, introduced by Johnson in 1972, see \cite{14}. For a Banach $A$-bimodule $B$, the quotient space $H^1(A,B)$ of all continuous derivations from $A$ into $B$ modulo the subspace of inner derivations is called the first cohomology group of $A$ with coefficients in $B$. Following \cite{ru.1} the Banach algebra $A$ is called super-amenable if $H^1(A,B)=\{0\}$ for every Banach $A$-bimodule $B$ (super-amenable Banach algebras are called contractible, too).  It is clear that if $A$ is super-amenable, then $A$ is  amenable.

In \cite{JKR},  Johnson,  Kadison, and  Ringrose introduced the notion of amenability for von Neumann algebras. The basic concepts, however, make sense for arbitrary dual Banach algebras. But is most commonly associated with  Connes, see \cite{A.Connes}. For this reason, this notion of amenability is called Connes-amenability (the origin of this name seems to be Helemskii, see \cite{A.Ya}).

Let $A$ be a Banach algebra. A Banach $A$-bimodule $X$ is called dual if there is a closed submodule $X_*$ of $X^*$ such that $X = (X_*)^*$ ($X_*$ is called the predual of $X$). A Banach algebra $A$ is called dual if it is dual as a Banach $A$-bimodule.

Let $A$ be a dual Banach algebra. A dual Banach $A$-bimodule $X$ is called normal if, for every $x\in X$, the maps
$$A\longrightarrow X,\hspace{0.2cm} a\mapsto\left\{
\begin{array}{ll}
a\cdot x &  \\
x\cdot a&
\end{array}
\right.$$
are weak$^*$-{continuous} ($w^*$-continuous). The  dual Banach algebra $A$ is called Connes-amenable if, for every dual Banach $A$-bimodule $X$, every $w^*$-continuous derivation $D:A\longrightarrow X$ is inner; or equivalently, ${H}_{w^*}^1(A,X) = \{0\}$ \cite{ru.1}.

The second dual $A^{**}$ of Banach algebra $A$ endowed with the either Arens multiplications is a Banach algebra. The constructions of the two Arens multiplications in $A^{**}$ lead us to the definition of topological centers for $A^{**}$ with respect to both Arens multiplications. The topological centers of Banach algebras, module actions and applications of them  were introduced and discussed in \cite{1, 12a, 13}.   To state our results, we need to fix some notations and recall some definitions.

Let $X,Y,Z$ be normed spaces and $m:X\times Y\rightarrow Z$ be a bounded bilinear mapping. Arens in \cite{1} offers two natural extensions $m^{***}$ and $m^{t***t}$ of $m$ from $X^{**}\times Y^{**}$ into $Z^{**}$, for more information see \cite{eshaghi,  12a, 13}

The mapping $m^{***}$ is the unique extension of $m$ such that $x^{\prime\prime}\rightarrow m^{***}(x^{\prime\prime},y^{\prime\prime})$ from $X^{**}$ into $Z^{**}$ is $weak^*-weak^*$ continuous for every $y^{\prime\prime}\in Y^{**}$, but the mapping $y^{\prime\prime}\rightarrow m^{***}(x^{\prime\prime},y^{\prime\prime})$ is not in general $weak^*-weak^*$ continuous from $Y^{**}$ into $Z^{**}$ unless $x^{\prime\prime}\in X$. Hence the first topological center of $m$ may be defined as follows
$$Z_1(m)=\{x^{\prime\prime}\in X^{**}:~~y^{\prime\prime}\rightarrow m^{***}(x^{\prime\prime},y^{\prime\prime})~~\text{is weak}^*\text{-weak}^*~~\text{continuous}\}.$$

Now, let  $m^t:Y\times X\rightarrow Z$ be the transpose of $m$ defined by $m^t(y,x)=m(x,y)$ for every $x\in X$ and $y\in Y$. Then $m^t$ is a continuous bilinear map from $Y\times X$ to $Z$, and so it may be extended as above to $m^{t***}:Y^{**}\times X^{**}\rightarrow Z^{**}$.
The mapping $m^{t***t}:X^{**}\times Y^{**}\rightarrow Z^{**}$ in general is not equal to $m^{***}$, see \cite{1}, if $m^{***}=m^{t***t}$, then $m$ is called Arens regular. The mapping $y^{\prime\prime}\rightarrow m^{t***t}(x^{\prime\prime},y^{\prime\prime})$ is $weak^*-weak^*$ continuous for every $x^{\prime\prime}\in X^{**}$, but the mapping $x^{\prime\prime}\rightarrow m^{t***t}(x^{\prime\prime},y^{\prime\prime})$ from $X^{**}$ into $Z^{**}$ is not in general  weak$^*$-weak$^*$ continuous for every $y^{\prime\prime}\in Y^{**}$. So we define the second topological center of $m$ as
$$Z_2(m)=\{y^{\prime\prime}\in Y^{**}:~~x^{\prime\prime}\rightarrow m^{t***t}(x^{\prime\prime},y^{\prime\prime})~~\text{is weak}^*\text{-weak}^*~~\text{continuous}\}.$$

It is clear that $m$ is Arens regular if and only if $Z_1(m)=X^{**}$ or $Z_2(m)=Y^{**}$. Arens regularity of $m$ is equivalent to the following
$$\lim_i\lim_j\langle  z^\prime,m(x_i,y_j)\rangle=\lim_j\lim_i\langle  z^\prime,m(x_i,y_j)\rangle,$$
whenever both limits exist for all bounded sequences $(x_i)_i\subseteq X$ , $(y_i)_i\subseteq Y$ and $z^\prime\in Z^*$, see \cite{ Dales}.

The mapping $m$ is left strongly Arens irregular if $Z_1(m)=X$ and $m$ is right strongly Arens irregular if $Z_2(m)=Y$.
The first Arens product is defined as follows in three steps. For $a,b$ in $A$,  $f$ in  $A^{*}$ and $m,n$ in $A^{**}$, the elements $f. a$, $m. f$ of $A^{*}$ and $m. n$ of $A^{**}$ are defined as follows:
$$\langle f. a,b\rangle=\langle f,ab\rangle,\quad\langle m. f,a\rangle=\langle m,f. a\rangle,\quad\langle m. n,f\rangle=\langle m,n. f\rangle.$$

The second Arens product is defined as follows. For $a,b$ in $A$,  $f$ in  $A^{*}$ and $m,n$ in $A^{**}$, the elements $ a\Diamond f$ , $f\Diamond m$ of $A^{*}$ and $m\Diamond n$ of $A^{**}$ are defined by the equalities
$$\langle a\Diamond f,b\rangle=\langle f,ba\rangle,\quad\langle f\Diamond m,a\rangle=\langle m,a\Diamond f\rangle,\quad\langle m\Diamond n,f\rangle=\langle n,f\Diamond m\rangle.$$

The Arens regularity of a normed algebra $A$ is defined to be the Arens regularity of its algebra multiplication when considered as a bilinear mapping $m:A\times A\rightarrow A$.
Let  $B$ be a Banach $A$-bimodule, and let
$$\pi_\ell:~A\times B\longrightarrow B\quad\text{and}\quad\pi_r:~B\times A\longrightarrow B,$$
be the right and left module actions of $A$ on $B$. By above notation, the transpose of $\pi_r$ denoted by $\pi_r^t:A\times B\rightarrow B$. Then
$$\pi_\ell^*:B^*\times~A\longrightarrow B^*\quad\text{and}\quad\pi_r^{t*t}: A\times B^*\longrightarrow B^*.$$
Thus $B^*$ is a left Banach $A$-module and a right Banach $A$-module with respect to the module actions $\pi_r^{t*t}$ and $\pi_\ell^*,$ respectively. The  second dual
$B^{**}$ is a Banach $A^{**}$-bimodule with the following  module actions
$$\pi_\ell^{***}:~A^{**}\times B^{**}\longrightarrow B^{**}\quad\text{and}\quad\pi_r^{***}:~B^{**}\times A^{**}\longrightarrow B^{**},$$
where $A^{**}$ is considered as a Banach algebra with respect to the first Arens product.
Similarly, $B^{**}$ is a Banach $A^{**}$-bimodule with the module actions
$$\pi_\ell^{t***t}:~A^{**}\times B^{**}\longrightarrow B^{**}\quad\text{and}\quad\pi_r^{t***t}:~B^{**}\times A^{**}\longrightarrow B^{**},$$
where $A^{**}$ is considered as a Banach algebra with respect to the second  Arens product. In this way we write $Z(\pi_{\ell})=Z_{B^{**}}(A^{**})$ and $Z(\pi_{r})=Z_{A^{**}}(B^{**})$.

Let  $B$ be a Banach $A$-bimodule. Then we say that  $B$  factors on the left (right) with respect to $A$, if $B=BA~(B=AB)$. Thus $B$ factors on both sides, if $B=BA=AB$.


\section{Weak$^*$-weak$^*$ continuous derivations  }

Let $B$ be a Banach $A$-bimodule. In this section, we study the cohomological properties of Banach algebra $A$ whenever every derivation in $Z^1(A^{**},B^*)$ is weak$^*$-weak$^*$ continuous.

\begin{theorem}\label{3.1a}
	Let $B$ be a Banach $A$-bimodule and let every derivation $D:A^{**}\longrightarrow B^*$ is weak$^*$-weak$^*$ continuous. If $Z^\ell_{B^{**}}(A^{**})=A^{**}$ and $H^1(A,B^*)=\{0\}$,  then  $H^1(A^{**},B^*)=\{0\}$.
\end{theorem}
\begin{proof}
	Let $D:A^{**}\longrightarrow B^*$ be a derivation. Then $D\mid_A:A\rightarrow B^*$ is a derivation. Since $H^1(A,B^*)=\{0\}$, there exists $b^\prime\in B^*$ such that $D\mid_A=\delta_{b^\prime}$. Suppose that $a^{\prime\prime}\in A^{**}$ and
	$(a_\alpha^{})_\alpha\subseteq A^{}$ such that $a_\alpha^{} \stackrel{w^*} {\longrightarrow} a^{\prime\prime}$ in $A^{**}$. Then
	
	\begin{align*}
	D(a^{\prime\prime})&=w^*-\lim_\alpha D\mid_A(a_\alpha)\\
	&=w^*-\lim_\alpha\delta_{b^\prime}(a_\alpha)\\
	&=w^*-\lim_\alpha(a_\alpha {b^\prime}-{b^\prime}a_\alpha)\\
	&=a^{\prime\prime}{b^\prime}-{b^\prime}a^{\prime\prime}.
	\end{align*}
	
	We now show that $b^\prime a^{\prime\prime}\in B^*$. Assume that  $(b^{\prime\prime}_{\beta})_{\beta}\in B^{**}$ such that $b^{\prime\prime}=w^*-\lim_{\beta}b^{\prime\prime}_{\beta}$. Since $Z^\ell_{B^{**}}(A^{**})=A^{**}$, we have
	$$\langle b^\prime a^{\prime\prime},b^{\prime\prime}_{\beta}\rangle=\langle  a^{\prime\prime}.b^{\prime\prime}_{\beta},b^\prime\rangle\rightarrow
	\langle  a^{\prime\prime}.b^{\prime\prime},b^\prime\rangle=\langle b^\prime a^{\prime\prime},b^{\prime\prime}\rangle.$$
	
	Thus, $b^\prime a^{\prime\prime}\in (B^{**},weak^*)^*=B^*$, and so $H^1(A^{**},B^*)=\{0\}$.
\end{proof}
\begin{corollary}\label{c1}
	Let $A$ be an Arens regular  Banach algebra and let every derivation $D:A^{**}\rightarrow A^*$ is weak$^*$-weak$^*$ continuous. If $A$ is weakly amenable,  then  $H^1(A^{**},A^*)=\{0\}$.
\end{corollary}
By the following result, we show that weak amenability of the Banach algebra $A$ is essential in vanishing of $H^1(A^{**},A^*)$.
\begin{proposition}\label{p1}
	Let $A$ be a  Banach algebra such that is an ideal in $A^{**}$. If $A$ is not  weakly amenable, then $H^1(A^{**},A^*)\neq\{0\}$.
\end{proposition}
\begin{proof}
	Let $d:A\longrightarrow A^*$ be a derivation and $\pi:A^{**}\longrightarrow A$ be a bounded homomorphism. Now; define $D:=d\circ\pi:A^{**}\longrightarrow A^*$. Clearly, $D$ is a bounded derivation which it is not inner. This shows that $H^1(A^{**},A^*)\neq\{0\}$.
\end{proof}
\begin{example}\label{ex1}
	\begin{itemize}
		\item[(i)] Let $K$ be a compact metric space, $d$ be a metric on $K$ and $\alpha\in(0,1]$. The Lipchitz algebra $\mathrm{Lip}_\alpha K$ is the space of complex-valued functions $f$ on $K$ such that
		$$p_\alpha(f)=\sup\left\{\frac{|f(x)-f(y)|}{d(x,y)^\alpha}:x,y\in K, x\neq y\right\}$$
		is finite. A subspace of  $\mathrm{Lip}_\alpha K$ that contains $f\in\mathrm{Lip}_\alpha K$ such that
		$$\frac{|f(x)-f(y)|}{d(x,y)^\alpha}\to 0\quad\text{as}\quad d(x,y)\to0$$
		is denoted by $\mathrm{lip}_\alpha K$. Let $\alpha\in(0,\frac{1}{2})$. Then by \cite[Theorem 4.4.34]{3Dales} or \cite[Theorem 3.8]{BCD}, $\mathrm{lip}_\alpha K$ is Arens regular and by \cite[Theorem 3.10]{BCD} it is weakly amenable. Then by Corollary \ref{c1}, $H^1\left(\left(\mathrm{lip}_\alpha K\right)^{**},\left(\mathrm{lip}_\alpha K\right)^*\right)=\{0\}$.
		\item[(ii)] Let $\omega$ be a weight sequence on $\mathbb{Z}$ such that
		$$\sup\left\{\frac{\omega(m+n)}{\omega(m)\omega(n)}\left(\frac{1+|n|}{1+|m+n|}\right):m,n\in\mathbb{Z}\right\}$$
		is finite. The Beurling algebra $\ell^1(\mathbb{Z},\omega)$ is not weakly amenable \cite[Theorem 2.3]{BCD}. Then by Proposition \ref{p1}, we have
		
		 $H^1\left(\ell^1(\mathbb{Z},\omega)^{**},\ell^\infty(\mathbb{Z},\omega)\right)\neq\{0\}$.
	\end{itemize}
\end{example}
Let $B$  be a dual Banach algebra, with predual  $X$ and suppose that
$$X^\perp=\{x^{\prime\prime\prime}:~x^{\prime\prime\prime}\mid_X=0~~\text{where}~~x^{\prime\prime\prime}\in X^{***}\}=\{b^{\prime\prime}:~b^{\prime\prime}\mid_X=0~~\text{where}~~b^{\prime\prime}\in B^{**}\}.$$

Then the canonical projection $P:X^{***}\longrightarrow X^*$ gives a continuous linear map $P:B^{**}\longrightarrow B$. Thus, we can write the following equality
$$B^{**}=X^{***}=X^*\oplus \ker P=B\oplus X^\perp ,$$
as a direct sum of Banach $A$-bimodules.
\begin{theorem}\label{3.3aa}
	Let $B$ be a Banach $A$-bimodule  such that  every derivation from $A^{**}$ into  $ B$ is weak$^*$-weak continuous and $A^{**}B,BA^{**}\subseteq B$.
	\begin{itemize}
		\item[(i)] If $H^1(A,B)=0$, then $H^1(A^{**},B)=\{0\}$.
		\item[(ii)]  Suppose that $A$ has a left bounded approximate identity (=LBAI), $B$ has a predual $X$ and  $AB^*,~B^*A\subseteq X$. If $H^1(A,B)=0$, then  $H^1(A^{**},B^{**})=\{0\}$.
	\end{itemize}
\end{theorem}
\begin{proof}
	(i)  Proof is similar to the proof of Theorem \ref{3.1a}.
	
	(ii) Set $B^{**}=B\oplus X^\perp .$ Then we have
	$$H^1(A^{**},B^{**})=H^1(A^{**},B)\oplus H^1(A^{**},X^\perp).$$
	
	Since $H^1(A,B)=\{0\}$, by  (i), $H^1(A^{**},B)=\{0\}$. Now let $\widetilde{D}\in Z^1(A^{**},X^\perp)$ and we take $D=\widetilde{D}\mid_A$. It is clear that $D\in Z^1(A^{**},X^\perp)$. Assume that $a^{\prime\prime}, x^{\prime\prime}\in A^{**}$ and $(a_\alpha)_\alpha, (x_\beta)_\beta\subseteq A$ such that $a_\alpha^{} \stackrel{w^*} {\rightarrow} a^{\prime\prime}$ and $x_\beta \stackrel{w^*} {\rightarrow} x^{\prime\prime}$ on $A^{**}$. Since $AB^*,~B^*A\subseteq X$, for every $b^ \prime \in B^*$, by using the weak$^*$-weak continuity  of $\widetilde{D}$, we have
	\begin{align*}
	\langle \widetilde{D}(a^{\prime\prime}\Diamond x^{\prime\prime}), b^\prime\rangle&=\lim_\beta\lim_\alpha \langle D(a_\alpha x_\beta), b^\prime\rangle\\
	&=\lim_\beta\lim_\alpha \langle (D(a_\alpha) x_\beta +a_\alpha D(x_\beta)),b^\prime\rangle\\
	&=\lim_\beta\lim_\alpha \langle D(a_\alpha) x_\beta ,b^\prime\rangle+\lim_\beta\lim_\alpha \langle a_\alpha D(x_\beta),b^\prime\rangle\\
	&=\lim_\beta\lim_\alpha \langle D(a_\alpha), x_\beta b^\prime\rangle+\lim_\beta\lim_\alpha \langle D(x_\beta)),b^\prime a_\alpha \rangle\\
	&=0.
	\end{align*}
	
	Since $A$ has a LBAI, $A^{**}$ has a left unit $e^{\prime\prime}$ with respect to the second Arens product \cite[Proposition 2.9.16]{3Dales}. Then  $D(x^{\prime\prime})= D(e^{\prime\prime}\Diamond x^{\prime\prime})=0$, and so $D=0$.
\end{proof}
\begin{example}
	\begin{itemize}
		\item[(i)] Assume that $G$ is a compact group. Then we know that $L^1(G)$ is $M(G)$-bimodule  and $L^1(G)$ is an ideal in the second dual of $M(G)$, $M(G)^{**}$. By  \cite[Corollary 1.2]{12}, we have $H^1(L^1(G),M(G))=\{0\}$. Then by Theorem \ref{3.3aa}, every weak$^*$-weak continuous derivation from $L^1(G)^{**}$ into $M(G)$ is inner.
		\item[(ii)]  We know that  $c_0$ is  a C$^*$-algebra and every C$^*$-algebra is weakly amenable,  so $c_0$ is weakly amenable. Then by Theorem \ref{3.3aa}, every weak$^*$- weak  continuous derivation from $\ell^\infty$ into $\ell^1$ is inner.
	\end{itemize}
\end{example}
\begin{theorem}\label{t2}
	Let $B$ be a Banach $A$-bimodule and $A$ has a $LBAI$. Suppose that $AB^{**},~B^{**}A\subseteq B$ and  every derivation from $A^{**}$ into  $B^*$ is {weak}$^*$-{weak}$^*$ continuous. If $H^1(A,B^*)=\{0\}$,  then  $H^1(A^{**},B^{***})=\{0\}$.
\end{theorem}
\begin{proof} Take $B^{***}=B^*\oplus B^\perp$, where $B^\perp=\{b^{\prime\prime\prime}\in B^{***}:~b^{\prime\prime\prime}\mid_B=0\}$. Then we have
	$$H^1(A^{**},B^{***})=H^1(A^{**},B^*)\oplus H^1(A^{**},B^\perp).$$
	
	Since $H^1(A,B^*)=\{0\}$,  similar to Theorem \ref {3.3aa}(i), we have $H^1(A^{**},B^*)=\{0\}$. It suffices to show that $H^1(A^{**},B^\perp)=0$.  	Let $(e_{\alpha})_{\alpha}\subseteq A$ be a LBAI for $A$ such that  $e_{\alpha} \stackrel{w^*} {\rightarrow}e^{\prime\prime}$ in $A^{**}$ where $e^{\prime\prime}$ is a left unit for  $A^{**}$ with respect to the second Arens product. Let $a^{\prime\prime}\in A^{**}$ and suppose that $(a_{\beta})_{\beta}\subseteq A$ such that  	$a_{\beta} \stackrel{w^*} {\rightarrow}a^{\prime\prime}$ in $A^{**}$. Let $D\in Z^1(A^{**},B^\perp)$. Then for every $b^{\prime\prime}\in B^{**}$, by  {weak}$^*$-{weak}$^*$ continuity  of $D$, we have
	\begin{align*}
	\langle D(a^{\prime\prime}), b^{\prime\prime}\rangle &= \langle D(e^{\prime\prime}\Diamond a^{\prime\prime}), b^{\prime\prime}\rangle\\
	&=\lim_\beta\lim_\alpha \langle (D(e_\alpha a_\beta),b^{\prime\prime}\rangle\\
	&=\lim_\beta\lim_\alpha \langle (D(e_\alpha) a_\beta +e_\alpha D(a_\beta)),b^{\prime\prime}\rangle\\
	&=\lim_\beta\lim_\alpha \langle D(e_\alpha) a_\beta ,b^{\prime\prime}\rangle+\lim_\beta\lim_\alpha \langle e_\alpha D(a_\beta),b^{\prime\prime}\rangle\\
	&=\lim_\beta\lim_\alpha \langle D(e_\alpha), a_\beta b^{\prime\prime}\rangle+\lim_\beta\lim_\alpha \langle  D(a_\beta),b^{\prime\prime}e_\alpha \rangle\\
	&=0.
	\end{align*}
	
	It follows that $D=0$, and so the result  holds.
\end{proof}
It is known that neither the weak amenability of $A$ implies that of $A^{**}$, nor the weak amenability of $A^{**}$ implies that of $A$. The question ``\emph{when the weak amenability of $A^{**}$ implies that of $A$?}'' is investigated in many works; see \cite{BHJ, Dales, EF, GL, f.ghahramani} for more details. We now by Theorem \ref{t2} consider the converse of the above question, i.e., ``\emph{under which conditions the weak amenability of $A$ implies that of $A^{**}$?}'', as follows:
\begin{corollary}\label{c2}
	Assume that $A$ is a Banach algebra with $LBAI$  such that it is two-sided ideal in $A^{**}$  and every derivation  $D:A^{**}\rightarrow A^{***}$ is weak$^*$- weak$^*$ continuous. If $A$ is weakly amenable, then $A^{**}$ is weakly amenable.
\end{corollary}
\begin{example}
	Assume that $G$ is a locally compact group. We know that $L^1(G)$ is weakly amenable Banach algebra, see \cite{Johnson}. Then by Corollary \ref{c2},  every weak$^*$- weak$^*$ continuous derivation from $L^1(G)^{**}$ into $L^1(G)^{***}$ is inner.
\end{example}

\begin{theorem}
	Let  $A$  be an amenable and Arens regular Banach algebra.  If for any normal Banach $A$-bimodule $B$ with predual $X$, we have $AB^*,~B^*A\subseteq X$, then $H^1_{w^*}(A^{**},B^{**})=\{0\}$.
\end{theorem}
\begin{proof}
	If the Banach algebra $A$ is amenable and Arens regular, then $A^{**}$ is Connes-amenable and the converse holds whenever $A$ is an ideal in $A^{**}$, too \cite[Theorem 4.4.8]{ru.1}. Thus  $H^1_{w^*}(A^{**},B)=\{0\}$ and by the argument before Theorem \ref {3.3aa}, we have $B^{**}=B\oplus X^\bot$. These imply that  $H^1_{w^*}(A^{**},B^{**})=H^1_{w^*}(A^{**},X^\bot)$. It is known that every amenable Banach algebra possesses a BAI, so by a similar argument in the proof of Theorem \ref {3.3aa}(ii), we obtain that $H^1_{w^*}(A^{**},X^\bot)=\{0\}$.
\end{proof}
\begin{proposition}\label{p2}
	Suppose that $A$ is  an amenable  Banach algebra. If for every Banach $A$-bimodule $B$, we have $AB^{**},~B^{**}A\subseteq B$, then
	\begin{align*}
	H^1_{w^*}(A^{**},B^{***})=\{0\}.
	\end{align*}
\end{proposition}
\begin{proof}
	By applying a similar argument in the proof of Theorem \ref {3.3aa}(ii), we obtain the desire.
\end{proof}
\begin{corollary}\label{c3}
	Assume that $A$ is a weakly amenable Banach algebra with a LBAI. If $A$ is an ideal in $A^{**}$, it follows that
	$$H^1_{w^*}(A^{**},A^{***})=\{0\}.$$
\end{corollary}
\begin{example}
	Assume that $G$ is a compact group. It is known that $L^1(G)$ has a BAI and  is a two-sided ideal in $L^1(G)^{**}$.  We know that $L^1(G)$ is weakly amenable, hence by Corollary \ref{c3},  $$H^1_{w^*}(L^\infty(G)^*,L^\infty(G)^{**})=\{0\}.$$
\end{example}
\begin{proposition}\label{3.1}
	Let $A$ be a Banach algebra such that $A$ is an ideal in $A^{**}$ and $A^*$ factors. Then $A$ is amenable if and only if $A^{**}$ is Connes-amenable.
\end{proposition}
\begin{proof}
	By \cite[Corollary 2.8]{BHJ}(i), $A$ is Arens regular. Then by   \cite[Theorem 4.4]{run.1}, the proof completes.
\end{proof}
A Banach space $A$ is called weakly sequentially complete if  every weakly Cauchy sequence in $A$ has a weak limit in $A$.
\begin{theorem}
	Let $A$ be an Arens regular dual Banach algebra such that $A^{*}$ is weakly sequentially complete (WSC). If $H^1_{w^*}(A^{**},A^{***})=\{0\}$, then $H^1_{w^{*}}(A,A^{*})=\{0\}$.
\end{theorem}
\begin{proof}
	Let $D:A\longrightarrow A^{*}$ be a $w^*$-continuous derivation. Since $A^{*}$ is WSC, every derivation $D:A\longrightarrow A^{*}$ is weakly compact. Then by   \cite[Theorem 6.5.5]{conway}, we have  $D''(A^{**})\subseteq A^{*}$  and hence, by Arens regularity of $A$, $A^{*}$ is an $A^{**}$-submodule of $(A^{**})^{*}$ and $D''(A^{**}).A^{**}\subseteq A^{*}.A^{**}\subseteq A^{*}$. Then by \cite[Theorem 7.1]{Dales}, $D'':A^{**}\longrightarrow A^{***}$ is a $w^{*}$-continuous derivation. Thus, there exists $a^{\prime\prime\prime}\in A^{***}$ such that $D''(F)=F.a^{\prime\prime\prime}-a^{\prime\prime\prime}.F$, for each $F\in A^{**}$. Now, let $E:A\longrightarrow A^{**}$ be the canonical map and set $f=E^{*}(a^{\prime\prime\prime})$, then $D(a)=a.f-f.a$, for all $a\in A$. This means that $D$ is an inner $w^*$-continuous derivation. Thus the proof follows.
\end{proof}
In the following, we extend the \cite[Corollary 2.8]{BHJ}(i)  to the general case as follows:
\begin{lemma}\label{3.3a}
	If $A^{(2n)}$ is a two-sided ideal in $A^{(2n+2)}$ and $A^{(2n+1)}$ factors, then $A^{(2n)}$ is Arens regular, where $n\in\mathbb{N}\cup\{0\}$.
\end{lemma}
\begin{theorem}\label{3.5}
	Let $A$ be a Banach algebra such that $A^{**}$ is an ideal in $A^{****}$ and $A^{***}$ factors. If $A$ is weakly amenable, then 	$H^1_{w^{*}}(A^{**},A^{***})=\{0\}$.
\end{theorem}
\begin{proof}
	Lemma \ref{3.3a} implies that $A^{**}$ is Arens regular. Now, let $D:A^{**}\longrightarrow A^{***}$ be a weak$^*$- weak$^*$-continuous derivation. First, we prove that $A^{***}$ is a normal Banach $A^{**}$-bimodule. Let $(a^{\prime\prime}_{\alpha})_{\alpha}$ be a net in $A^{**}$ and $a^{\prime\prime\prime}\in A^{***}$. Then, by Arens regularity of $A^{**}$, for every $b^{\prime\prime}\in A^{**}$ we have
	\begin{align*}
	\langle(w^*-\lim_{\alpha}a^{\prime\prime}_{\alpha}).a^{\prime\prime\prime},b^{\prime\prime}\rangle&=\langle a^{\prime\prime\prime},b^{\prime\prime} .(w^*-\lim_{\alpha}a^{\prime\prime}_{\alpha})\rangle\\&=\lim_{\alpha}\langle a^{\prime\prime\prime},b^{\prime\prime}.a^{\prime\prime}_{\alpha}\rangle\\&=\lim_{\alpha}\langle a^{\prime\prime}_{\alpha}.a^{\prime\prime\prime},b^{\prime\prime}\rangle\\&=\langle w^*-\lim_{\alpha}(a^{\prime\prime}_{\alpha}.a^{\prime\prime\prime}),b^{\prime\prime}\rangle.
	\end{align*}
	
	Moreover,
	\begin{align*}
	\langle a^{\prime\prime\prime}.(w^*-\lim_{\alpha}a^{\prime\prime}_{\alpha}),b^{\prime\prime}\rangle&=\langle a^{\prime\prime\prime},w^*-\lim_{\alpha}a^{\prime\prime}_{\alpha}.b^{\prime\prime}\rangle\\&=\lim_{\alpha}\langle a^{\prime\prime\prime},a^{\prime\prime}_{\alpha},b^{\prime\prime}\rangle\\&=\lim_{\alpha}\langle a^{\prime\prime\prime}.a^{\prime\prime}_{\alpha},b^{\prime\prime}\rangle\\&=\langle w^*-\lim_{\alpha}(a^{\prime\prime\prime}.a^{\prime\prime}_{\alpha}),b^{\prime\prime}\rangle.
	\end{align*}
	
	Hence, the mapping $a^{\prime\prime}\mapsto a^{\prime\prime}.a^{\prime\prime\prime}$ and $a^{\prime\prime}\mapsto a^{\prime\prime\prime}.a^{\prime\prime}$ are weak$^*$-weak$^*$-continuous from $A^{**}$ into $ A^{***}$. Thus, $A^{***}$ is a normal Banach $A^{**}$-bimodule. For each $a\in A$, we define $\bar{D}:A\longrightarrow A^*$ by
	$$\bar{D}(a)=D(\widehat{a})\mid _A,$$
	where $\widehat{a}\in A^{**}$ with  $\widehat{a}(a^\prime )=a^\prime (a)$, for all $a\in A$.	
	As the following equalities   $\bar{D}$ is a continuous derivation from $A$ into $A^*$.
	\begin{center}
		$\bar{D}(ab)=D(\widehat{ab})=D(\widehat{a}.\widehat{b})=a.D(\widehat{b})+D(\widehat{a}).b=a.\bar{D}(b)+\bar{D}(a).b$,
	\end{center}
	where   $a,b\in A$.
	By weak amenability of $A$, we have $\bar{D}$ is inner. Then there exist $a^{\prime\prime\prime}\in A^{***}$ such that
	\begin{center}
		$D(\widehat a)=\bar{D}(a)=a.a^{\prime\prime\prime}\mid_A-
		a^{\prime\prime\prime}\mid_A.a=\widehat{a}.a^{\prime\prime\prime}\mid_A-a^{\prime\prime\prime}\mid_A.\widehat{a}.$
	\end{center}
	
	We consider the canonical mapping $E:A^*\longrightarrow A^{***}$. Then there exists $b^{\prime\prime\prime}\in A^{***}$ such that $E(a^{\prime\prime\prime}\mid_A)=b^{\prime\prime\prime}$. So
	\begin{center}
		$D(\widehat{a})=\widehat{a}.b^{\prime\prime\prime}-b^{\prime\prime\prime}.\widehat{a}.$
	\end{center}
	
	Then  $D$ is inner. It follows that  $H^1_{w^{*}}(A^{**},A^{***})=\{0\}$.
\end{proof}
\begin{corollary}
	Let $A^{(2n+2)}$ be a two sided ideal in $A^{(2n+4)}$ and $A^{(2n+3)}$ factors. If $A^{(2n)}$ is weakly amenable, then $H^1_{w^*}(A^{(2n+2)},A^{(2n+3)})=\{0\}$.
\end{corollary}
\begin{proof}
	Apply Lemma  $\ref{3.3a}$ and Theorem $\ref{3.5}$.
\end{proof}
Weak$^*$-continuous derivations from dual Banach algebras into their ideals are studied in \cite{EJ}.
\begin{remark}
	If $M$ is subspace of $A$ and $N$  is subspace of $A^*$, then $M^{\perp}=\{x^*\in X^*: \langle x^*,x\rangle=0,\quad \forall x\in M\}$ and $^{\perp}N=\{x\in A: \langle x^*,x\rangle=0,\quad\forall x^{*}\in N\}$. If $A$ is a dual Banach algebra and $I$ is $w^*-$closed ideal of $A$, then $I$ is dual with predual $I_*=\frac{A_{*}}{^{\perp}I}$ that $(I_{*})^{*}=(\frac{A_{*}}{^{\perp}{I}})^*=(^{\perp}I)^{\perp}=I$ and $I^*=\frac{A^*}{I^{\perp}}$, see \cite{conway}.
\end{remark}
\begin{proposition}\label{pro1}
	Let $A$ be a dual Banach algebra and $I$ be an arbitrary $w^*$-closed ideal of $A$ such that $H^1(A,I^{**})=\{0\}$. Then $H^1_{w^*}(A,I)=\{0\}$.
\end{proposition}
\begin{proof}
	Let $D\in Z^1_{w^*}(A,I)$ and $E:I\rightarrow I^{**}$ be the natural embedding. Then $E\circ D:A\rightarrow I^{**}$ is a bounded derivation. Since $H^1(A,I^{**})=\{0\}$,  there exists $a^{**}\in I^{**}$ such that $E\circ D=\delta_{a^{**}}$. Consider the decomposition $I^{**}=I\oplus I_{*}^{\perp}$ as an $A$-bimodule. If $P:I^{**}\rightarrow I$ is a projection, we have $D=\delta_{p(a^{**})}$. Then  $H^1_{w^*}(A,I)=\{0\}$.
\end{proof}
A  Banach algebra $A$ is without of order if for any $a,b\in A$, $ab=0$ implies that $a=0$ or $b=0$. Semisimple and unital Banach algebras are without of order Banach algebras. Now by Proposition \ref{pro1} and \cite[Theorem 3.1]{EJ}, we have the following result.
\begin{corollary}\label{cc1}
	Let $A$ be a dual Banach algebra and $I$  be a closed two-sided ideal in $A$ such that $I$ is without order. If $H^1(A,I^{**})=\{0\}$, then $H^1_{w^*}(I,I)=\{0\}$.
\end{corollary}
\begin{example}
	$\mathrm{(i)}$  Let $G$ be a locally compact group. A linear subspace $S^1(G)$ of $L^1(G)$ is said to be a Segal algebra, if it satisfies the following conditions:
	\begin{itemize}
		\item[(S1)] $S^1(G)$ is a dense  in $L^1(G)$;
		\item[(S2)] If $f\in S^1(G)$, then $L_xf\in S^1(G)$, i.e. $S^1(G)$ is left translation invariant;
		\item[(S3)] $S^1(G)$ is a Banach space under some norm $\|\cdot\|_S$ and $\|L_xf\|_s=\|f\|_s$, for all $f\in S^1(G)$ and $x\in G$;
		\item[(S4)]  $x\mapsto L_xf$ from $G$ into $S^1(G)$ is continuous.
	\end{itemize}
	
	For more details about Segal algebras, see \cite{ri1, ri}. Now, let $G$ be an abelian locally compact group. Then $H^1(L^1(G),S^1(G)^{**})=\{0\}$. Then by Proposition \ref{pro1} and Corollary \ref{cc1}, we have $H_{w^*}^1(L^1(G),S^1(G))=\{0\}$ and $H_{w^*}^1(S^1(G),S^1(G))=\{0\}$.
	
	$\mathrm{(ii)}$ Let $\Lambda$ be a non-empty, totally ordered set, and regard it as a semigroup by defining the product
	of two elements to be their maximum. The resulting semigroup, which we denote by $\Lambda_\vee$, is a
	semilattice. We may then form the $\ell^1$-convolution algebra $\ell^1(\Lambda_\vee)$. For every $t\in\Lambda_\vee$ we denote
	the point mass concentrated at $t$ by $e_t$. The definition of multiplication in $\ell^1(\Lambda_\vee)$ ensures that
	$e_se_t = e_{\max(s,t)}$ for all $s$ and $t$.
	
	The semilattice $\Lambda_\vee$, is a commutative semigroup in which every element is idempotent. If we denote the set of  idempotent elements of $\Lambda_\vee$ by $E(\Lambda_\vee)$, then $E(\Lambda_\vee)=\Lambda_\vee$.
	The $\ell^1$-convolution algebras of semilattices provide interesting examples of commutative Banach
	algebras. By \cite[Proposition 3.3]{J}, $H^1(\ell^1(\Lambda_\vee),I^{**})=\{0\}$, for any closed two-sided $I$ of $\ell^1(\Lambda_\vee)$. Then  by Corollary \ref{cc1},  $H^1_{w^*}(I,I)=\{0\}$,  for any closed two-sided $I$ of $\ell^1(\Lambda_\vee)$.
	
	$\mathrm{(iii)}$ Let $K$ be an infinite compact metric space,  $\alpha\in(0,1)$ and $\mathrm{lip}_\alpha K$ be the small Lipchitz algebra (see Example \ref{ex1}). By \cite[Proposition 3.4]{J}, $H^1(\mathrm{lip}_\alpha K,I^{**})=\{0\}$, for any closed two-sided $I$ of $\mathrm{lip}_\alpha K$. Then by Corollary \ref{cc1}, $H^1_{w^*}(I,I)=\{0\}$,  for any closed two-sided $I$ of $\mathrm{lip}_\alpha K$.
\end{example}
\section{Representations of derivations and Arens regularity }
Let $A$ be a Banach algebra and $B$ be a Banach $A$-bimodule with the module action ``$\bullet$''. Then for every $b\in B$, we define
$$L_{b}(a)=b\bullet a\quad\text{and}\quad R_{b}(a)=a\bullet b,$$
for every $a\in A$. These are the operation of left and right multiplication by $b$ on $A$. In the following by using the super-amenability of Banach algebra $A$, we give a representation for $Z^1(A,C)$, where $C$ is a Banach $A$-bimodule.

For a Banach $A$-bimodule $B$ and for a derivation $D:A\rightarrow B^{*}$, we show that the left module action $\pi_\ell:A\times B\rightarrow B$ is Arens regular whenever  $D^{\prime\prime}:A^{**}\rightarrow B^{***}$ is a derivation and $B^*\subseteq D^{\prime\prime}(A^{**})$.   On the other hand, if $A$ is a left strongly Arens irregular and $A^{**}$ is amenable Banach algebra with respect to the first Arens product, then $A$ is unital. Moreover, if $A$ is a dual Banach algebra, it follows that $A$ is reflexive.
\begin{theorem}\label{t3}
	Assume that $A$ is an amenable Banach algebra. Then there are Banach $A$-bimodules $C$, $D$ and  elements $\mathfrak{a}, \mathfrak{b}\in A^{**}$ such that $$Z^1(A,C^*)=\{R_{D^{\prime\prime}(\mathfrak{a})}:~D\in Z^1(A,C^*)\}=\{L_{D^{\prime\prime}(\mathfrak{b})}:~D\in Z^1(A,D^*)\}.$$
\end{theorem}
\begin{proof}
	Suppose that $B$ is a Banach $A$-bimodule with a module action $\bullet$. Every amenable Banach algebra has a BAI \cite[Proposition 2.2.1]{ru.1}, so $A$ has a BAI such as $(e_{\alpha})_{\alpha}$. Then by Cohen factorization Theorem we have $B\bullet A=B=A\bullet B$, i.e.,  for every $b\in B$, there are $y,z\in B$ and $a,t\in A$ such that $y\bullet a=b=t\bullet z$. Then we have
	\begin{equation}\label{eqeq1}
	\lim_\alpha b\bullet e_{\alpha}=\lim_\alpha (y\bullet a)\bullet e_{\alpha}=\lim_\alpha y\bullet(a e_{\alpha})=y\bullet a=b
	\end{equation}
	and
	\begin{equation}\label{eqeq11}
	\lim_\alpha e_{\alpha}\bullet b=\lim_\alpha e_{\alpha}\bullet (t\bullet z)=\lim_\alpha (te_{\alpha})\bullet z=t\bullet z=b.
	\end{equation}
	
	It follows that $B$ has a BAI as $(e_{\alpha})_{\alpha}\subseteq A$. Let $e^{\prime\prime}$ and $f''$ be the right and left unit for $A^{**}$, respectively such that $e_{\alpha} \stackrel{w^*} {\rightarrow}e^{\prime\prime}$ and $e_{\alpha} \stackrel{w^*} {\rightarrow}f^{\prime\prime}$ in $A^{**}$.
	
	Take $C=B$ and  define a module action ``$\cdot$'' as $a\cdot x=0$ and $x\cdot a=x\bullet a$, for all $a\in A$ and $x\in C$. Clearly, $(C,\cdot)$ is a Banach $A$-bimodule. Suppose that $D\in Z^1(A,C^*)$. Then there is an element $c\in C^*$ such that $D=\delta_c$. Then for every $a\in A$, we have
	$$D(a)=\delta_c (a)=a\cdot c-c\cdot a=a\bullet c.$$
	
	From \eqref{eqeq1} and module actions of $C$, for any $x\in C$ and $x'\in C^*$, we have
	\begin{equation}\label{eqeq2}
	\lim_\alpha\langle x,x'\cdot e_\alpha\rangle= \lim_\alpha\langle e_\alpha\cdot x,x' \rangle=0
	\end{equation}
	and
	\begin{equation}\label{eqeq3}
	\lim_\alpha\langle x,e_\alpha\cdot x' \rangle= \lim_\alpha\langle x\cdot e_\alpha,x' \rangle=\lim_\alpha\langle x\bullet e_\alpha,x' \rangle=\langle x, x' \rangle.
	\end{equation}
	
	It follows that $e_{\alpha}\cdot x' \stackrel{w^*}{\rightarrow}x'$ in $C^*$. Since $D^{\prime\prime}$ is a weak$^*$-to-weak$^*$ continuous linear operator, we have
	\begin{align*}
	D^{\prime\prime}(e^{\prime\prime})&=D^{\prime\prime}(w^*-\lim_\alpha e_{\alpha})=w^*-\lim_\alpha D^{\prime\prime}(e_{\alpha})=w^*-\lim_\alpha D(e_{\alpha})\\
	&=w^*-\lim_\alpha (e_{\alpha}x')=x'.
	\end{align*}
	
	Thus we conclude that $D(a)=a\cdot D^{\prime\prime}(e^{\prime\prime})=a\bullet D^{\prime\prime}(e^{\prime\prime})$ for all $a\in A$. It follows that $D=R_{D^{\prime\prime}(e^{\prime\prime})}.$
	On the other hand, since for every derivation $D\in Z^1(A,C^*)$,  $R_{D^{\prime\prime}(e^{\prime\prime})}\in
	Z^1(A,C^*)$, the result  holds.
	
	Now, again consider $B$ as a Banach $A$-bimodule with the module action ``$\bullet$'' and set $D=B$ with the module action $\triangleleft$ such that $a\triangleleft y=a\bullet y$ and $y\triangleleft a=0$, for all $a\in A$ and $y\in D$. By a similar argument that we have discussed above, and setting $\mathfrak{b}=f''$, the proof completes.
\end{proof}
\begin{example}
	\begin{itemize}
		\item[(i)]  Let $G$ be an amenable locally compact group. Then by Johnson Theorem $H^1(L^1(G),X^*)=\{0\}$, for every Banach $A$-bimodule $X$. Then by defining the similar module actions of $L^\infty(G)$ as a Banach $L^1(G)$-bimodule in the proof of Theorem \ref{t3} and by this Theorem, we have
		\begin{align*}
		Z^1(L^1(G),L^\infty(G))&=\{R_{D(e'')}:D\in Z^1(L^1(G),L^\infty(G))\}  \\
		& =\{L_{D(f'')}:D\in Z^1(L^1(G),L^\infty(G))\},
		\end{align*}
		where $e''$ and $f''$ are the left and right units of $L^1(G)^{**}$, indeed they $w^*$-accumulations of the BAI of $L^1(G)$.
		\item[(ii)]  Let $G$ be locally compact group. Then by \cite[Corollary 1.2]{12} $H^1(L^1(G),M(G))=\{0\}$. Then by applying the module actions defined in the proof of Theorem \ref{t3}, we can see  $M(G)$ as a Banach $L^1(G)$-bimodule. Then by Theorem \ref{t3}, we have \begin{align*}
		Z^1(L^1(G),L^\infty(G))&=\{R_{D(e'')}:D\in Z^1(L^1(G),L^\infty(G))\}  \\
		& =\{L_{D(f'')}:D\in Z^1(L^1(G),L^\infty(G))\},
		\end{align*}
		where $e''$ and $f''$ are the left and right units of $L^1(G)^{**}$.
	\end{itemize}
\end{example}
\begin{theorem}\label{t4}
	Let $A$ be a Banach algebra, $B$ be a Banach $A$-bimodule and  $D:A\longrightarrow B^*$ be a continuous derivation. If $D^{\prime\prime}:A^{**}\longrightarrow B^{***}$ is a derivation and $B^*\subseteq D^{\prime\prime}(A^{**})$, then $Z^\ell_{A^{**}}(B^{**})=B^{**}$.
\end{theorem}
\begin{proof}
	Since $D^{\prime\prime}:A^{**}\rightarrow B^{***}$ is a derivation, by \cite[Theorem 4.2]{20},
	$D^{\prime\prime}(A^{**})B^{**}\subseteq B^*$. Due to $B^*\subseteq D^{\prime\prime}(A^{**})$, we have $B^*B^{**}\subseteq B^*$. Let   $(a_{\alpha}^{\prime\prime})_{\alpha}\subseteq A^{**}$ such that  $a^{\prime\prime}_{\alpha} \stackrel{w^*} {\rightarrow}a^{\prime\prime}$ in $A^{**}$. Assume that $b^{\prime\prime}\in B^{**}$. Then for every $b^\prime\in B^*$, since $b^\prime b^{\prime\prime}\in B^*$, we have
	$$\langle b^{\prime\prime}.a_{\alpha}^{\prime\prime},b^\prime\rangle=\langle a_{\alpha}^{\prime\prime}b^\prime, b^{\prime\prime}\rangle\rightarrow
	\langle a^{\prime\prime},b^\prime b^{\prime\prime}\rangle=\langle b^{\prime\prime}.a^{\prime\prime},b^\prime \rangle.$$
	
	Thus $b^{\prime\prime}.a^{\prime\prime}_{\alpha} \stackrel{w^*} {\rightarrow}b^{\prime\prime}.a^{\prime\prime}$ is in $B^{**}$, and so $b^{\prime\prime}\in Z^\ell_{A^{**}}(B^{**})$.
\end{proof}
\begin{corollary}\label{c3.4}
	Let $A$ be a Banach algebra and $D:A\longrightarrow A^*$ be a continuous derivation such that $A^*\subseteq D^{\prime\prime}(A^{**})$. If  $D^{\prime\prime}:A^{**}\longrightarrow A^{***}$ is a derivation, then  $A$ is Arens regular.
\end{corollary}
\begin{example}
	Let $G$ be an infinite locally compact group. Thus, $L^1(G)$ is not Arens regular. Then Corollary \ref{c3.4} implies that there is no $D\in Z^1(L^1(G),L^1(G)^*)$ such that $L^1(G)^*\subseteq D^{\prime\prime}(L^1(G)^{**})$ and its second transpose $D''$ is a derivation.
\end{example}
\begin{lemma}\label{lem3}
	Let $B$ be a Banach left $A$-module and $B^{**}$ has a LBAI with respect to $A^{**}$. Then $B^{**}$ has a left unit with respect to $A^{**}$.
\end{lemma}
\begin{proof}
	Assume that $(e^{\prime\prime}_{\alpha})_{\alpha}\subseteq A^{**}$ is a LBAI for $B^{**}$. By passing to a suitable subnet, we may suppose that there is an $e^{\prime\prime}\in A^{**}$ such that $e^{\prime\prime}_{\alpha} \stackrel{w^*} {\rightarrow}e^{\prime\prime}$ in $A^{**}$. Then for every $b^{\prime\prime}\in B^{**}$ and $b^\prime\in B^*$, we have
	\begin{align*}
	\langle \pi_\ell^{***}(e^{\prime\prime},b^{\prime\prime}),b^\prime\rangle &=
	\langle e^{\prime\prime},\pi_\ell^{**}(b^{\prime\prime},b^\prime)\rangle=
	\lim_\alpha \langle e_\alpha^{\prime\prime},\pi_\ell^{**}(b^{\prime\prime},b^\prime)\rangle\\
	&=\lim_\alpha \langle \pi_\ell^{***}(e_\alpha^{\prime\prime},b^{\prime\prime}),b^\prime\rangle=
	\langle b^{\prime\prime},b^\prime\rangle.
	\end{align*}
	
	It follows that $\pi_\ell^{***}(e^{\prime\prime},b^{\prime\prime})=b^{\prime\prime}$.
\end{proof}
\begin{theorem}\label{t44}
	Let $A$ be a left strongly Arens irregular and suppose that $A^{**}$ is an amenable Banach algebra. Then we have the following assertions.
	\begin{itemize}
		\item[(i)] $A$ has an identity.
		\item[(ii)] If $A$ is a dual Banach algebra, then $A$ is reflexive.
	\end{itemize}
\end{theorem}
\begin{proof}
	(i) Amenability of $A^{**}$ implies that it has a BAI.  By using Lemma \ref{lem3}, $A^{**}$ has an identity say that $e^{\prime\prime}$. So, the mapping $x^{\prime\prime}\rightarrow  e^{\prime\prime}. x^{\prime\prime}=x^{\prime\prime}$ is weak$^*$-to-weak$^*$ continuous from $A^{**}$ into $A^{**}$. It follows that $e^{\prime\prime}\in Z_1(A^{**})=A$. This means that $A$ has an identity.
	
	(ii) Assume that $E$ is a predual of $A$. Then we have $A^{**}=A\oplus E^\bot$. Since $A^{**}$ is amenable, by  \cite[Theorem 1.8]{f.ghahramani} or \cite[Theorem 2.3]{G},  $A$ is amenable, and so $E^\bot$ is amenable. Thus $E^\bot$ has a $BAI$ such as $(e^{\prime\prime}_\alpha)_\alpha\subseteq E^\bot$. Since $E^\bot$ is a closed and weak$^*$-closed subspace of $A^{**}$, without loss generality,  there is $e^{\prime\prime}\in E^\bot$ such that
	$$e^{\prime\prime}_\alpha\stackrel{w^*} {\longrightarrow}e^{\prime\prime}\qquad \text{and}\qquad e^{\prime\prime}_\alpha\stackrel{\|\cdot\|} {\longrightarrow}e^{\prime\prime}.$$
	
	Then   $e^{\prime\prime}$ is a left identity for  $E^\bot$. On the other hand, for every $x^{\prime\prime}\in E^\bot$, since $E^\bot$ is an ideal in $A^{**}$, we have $x^{\prime\prime}.e^{\prime\prime}\in E^\bot$. Thus, for every $a^\prime\in A^*$,
	\begin{align*}
	\langle x^{\prime\prime}.e^{\prime\prime},a^\prime\rangle &=\lim_\alpha
	\langle (x^{\prime\prime}.e^{\prime\prime}).e^{\prime\prime}_\alpha,a^\prime\rangle=\lim_\alpha
	\langle x^{\prime\prime}.(e^{\prime\prime}.e^{\prime\prime}_\alpha),a^\prime\rangle\\
	&=\lim_\alpha
	\langle x^{\prime\prime}.e^{\prime\prime}_\alpha,a^\prime\rangle= \langle x^{\prime\prime},a^\prime\rangle.
	\end{align*}
	
	It follows that $x^{\prime\prime}.e^{\prime\prime}=x^{\prime\prime}$, and so  $e^{\prime\prime}$ is a right identity for $E^\bot$. Consequently, $e^{\prime\prime}$ is a two-sided identity for $E^\bot$.
	Now, let $a^{\prime\prime}\in A^{**}$. Then
	$$e^{\prime\prime}.a^{\prime\prime}= (e^{\prime\prime}.a^{\prime\prime}).e^{\prime\prime}=
	e^{\prime\prime}.(a^{\prime\prime}.e^{\prime\prime})=a^{\prime\prime}.e^{\prime\prime}.$$
	
	Hence $e^{\prime\prime}\in Z_1(A^{**})=A$. It follows that  $e^{\prime\prime}=0$, and so $E^\bot=0$. This implies that $A^{**}=A$.
\end{proof}
\begin{example}
	Let $G$ be a locally compact group. If  $M(G)^{**}$ is amenable, then by Theorem \ref{t44}(ii), because $C_0(G)^*=M(G)$, we conclude that $M(G)$ is reflexive. This means that $G$ is a finite group, moreover see \cite[Corollary 1.4]{f.ghahramani}.
\end{example}


\bigskip
\noindent
{\bf Acknowledgment.}
 We would like to thank the referee for her/his careful reading of our paper and many valuable suggestions.

\bibliographystyle{amsplain}

\vspace{0.1in}
\hrule width \hsize \kern 1mm
\end{document}